\documentclass[11pt]{article}

\usepackage[T1]{fontenc}
\usepackage[latin1]{inputenc}

\usepackage{amssymb}
\usepackage{amsthm}
\usepackage{amsmath}
\allowdisplaybreaks

\newcommand{\NN}{\mathbb{N}}

\newcommand{\RR}{\mathbb{R}}
\newcommand{\CC}{\mathbb{C}}

\newcommand{\HH}{\mathbb{H}}

\newtheorem{thm}{Theorem}[section]

\newtheorem{lem}[thm]{Lemma}
\newtheorem{prop}[thm]{Proposition}

\begin{document}

\title{Explicit fundamental solution for the operator $L+\alpha|T|$ on the Gelfand pair  $(\mathbb{H}_{n},U(n))$}

\author{I. E. Cardoso
  \and
  M. Subils
  \and
  R. E. Vidal
}

\newcommand{\Addresses}{{
  \bigskip
  \footnotesize

  I. E.~Cardoso, \textsc{ECEN-FCEIA, Universidad Nacional de Rosario, Pellegrini 250, 2000 Rosario, Santa Fe, Argentina}\par\nopagebreak
  \textit{E-mail address}, I. E.~Cardoso: \texttt{isolda@fceia.unr.edu.ar}

  \medskip

  M.~Subils, \textsc{ECEN-FCEIA, Universidad Nacional de Rosario, Pellegrini 250, 2000 Rosario, Santa Fe, Argentina}\par\nopagebreak
  \textit{E-mail address}, M.~Subils: \texttt{msubils@fceia.unr.edu.ar}

  \medskip

  R. E.~Vidal, \textsc{Fac. de Matemática, Astronomía y Física, Universidad Nacional de Córdoba \\
Medina Allende s/n, 5000 Córdoba, Córdoba, Argentina}\par\nopagebreak
  \textit{E-mail address}, R. E.~Vidal: \texttt{vidal@mate.uncor.edu.ar}

}}

\maketitle

\begin{abstract}
By means of the spherical functions associated to the Gelfand pair $(\mathbb{H}_{n},U(n))$ we define the operator $L+\alpha |T|$, where $L$ denotes the Heisenberg sublaplacian and $T$ denotes de central element of the Heisenberg Lie algebra, we establish a notion of fundamental solution and explicitly compute in terms of the Gauss hypergeometric function. For $\alpha<n$ we use the Integral Representation Theorem to obtain a more detailed expression. Finally, we remark that when $\alpha=0$ we recover the fundamental solution for the Heisenberg sublaplacian given by Folland.

\end{abstract}

\section{Introduction and preliminaries}

\par Let us consider the Heisenberg group $\HH_{n}=\CC^{n}\times\RR$ under the action of the unitary group $U(n)$, namely for $u\in U(n)$ and $(z,t)\in\mathbb{H}_{n}$ one has that $$u\cdot (z,t)=(u(z),t).$$ This group is a real Lie group, its Haar measure coincides with the Lebesgue measure in $\CC^{n}\times\RR$, and the metric we consider is the Carnot-Caratheodory metric given by \begin{align}\label{metric} |(z,t)| = & (|z|^{4}+16t^{2})^{\frac{1}{2}}, \end{align} for $(z,t)\in\HH_{n}$. Also, let $\{X_{1},\dots,X_{n},Y_{1},\dots,Y_{n},T\}$ be the canonical basis for the associated Lie algebra $\mathfrak{h}_{n}$, whose elements we can write in global coordinates as follows: for $j=1,\dots,n$
\begin{align*}
X_{j}=& \frac{\partial}{\partial x_{j}} - \frac{1}{2} y_{j} \frac{\partial}{\partial t},\\
Y_{j}=& \frac{\partial}{\partial y_{j}} + \frac{1}{2} x_{j} \frac{\partial}{\partial t},\\
T=& \frac{\partial}{\partial t}.
\end{align*}
Let us also consider the sublaplacian $L$: $$ L = \sum\limits_{j=1}^{n} X_{j}^{2}+Y_{j}^{2},$$ which in coordinates is given by
\begin{align*}
L=&  \sum\limits_{j=1}^{n} \left( \frac{\partial^{2}}{\partial x_{j}^{2}} +\frac{\partial^{2}}{\partial y_{j}^{2}}   \right) + \frac{1}{4}\sum\limits_{j=1}^{n}(x_{j}^{2}+y_{j}^{2})\frac{\partial^{2}}{\partial t^{2}} + \frac{\partial}{\partial t} \sum\limits_{j=1}^{n} \left( x_{j}\frac{\partial}{\partial y _{j}} - y_{j} \frac{\partial}{\partial x_{j}}\right) \\
= & \triangle_{\mathbb{R}^{2n}} +\frac{1}{4} |z|^{2}T^{2} + TR,
\end{align*}
where we define the operator $R$ as $$ R= \sum\limits_{j=1}^{n} \left( x_{j}\frac{\partial}{\partial y _{j}} - y_{j} \frac{\partial}{\partial x_{j}}\right).$$

\par It is a known fact that $L$ commutes with the action of $U(n)$ (see for example \cite{FH}), and of course also does $T$. Moreover, the operators $L$ and $T$ generate the subalgebra $\mathcal{U}(\mathfrak{h}_{n})^{U(n)}$ of the universal enveloping algebra $\mathcal{U}(\mathfrak{h}_{n})$ consisting of those left invariant differential operators which commute with the action of $U(n)$ (see for example \cite{MR}).

\par Associated to the Gelfand pair $(\HH_{n},U(n))$ there exists a family of \textit{spherical functions} $\{\varphi_{\lambda,k}\}_{\lambda\in\RR^{\ast},k\in\NN_{0}}$ which are smooth, bounded, $U(n)$-invariant and satisfy $\varphi_{\lambda,k}(0,0)=1$. For $\lambda\in\RR^{\ast}$ and $k\in\mathbb{N}_{0}$, each one is defined from the Schr\"{o}dinger representation $\pi_{\lambda}$ and the special Hermite functions $h_{\beta}^{\lambda}$, $|\beta|=k$, as follows (see for example \cite{T}):
\begin{align*}
\varphi_{\lambda,k}(z,t)=&\sum\limits_{|\beta|=k} E_{\lambda}(h_{\beta}^{\lambda},h_{\beta}^{\lambda})(z,t),
\end{align*}
where $E_{\lambda}(h_{\beta}^{\lambda},h_{\beta}^{\lambda})$ are the diagonal entries given by
\begin{align*}
E_\lambda(h_{\beta}^{\lambda},h_{\beta}^{\lambda})(z,t)=&<\pi_{\lambda}(z,t) h_{\beta}^{\lambda},h_{\beta}^{\lambda}>.
\end{align*}
Hence they can be computed, and they are explicitly given by (see for example \cite{BJR}, \cite{Str})
\begin{align}\label{spherical.functions.explicit}
\varphi_{\lambda,k}(z,t)=e^{i\lambda t}L_{k}^{n-1}\left( {\frac{|\lambda|}{2}} |z|^{2} \right) e^{-{\frac{|\lambda|}{4}}|z|^{2}},
\end{align}
where $\lambda\in\mathbb{R}^{\ast}=\mathbb{R}\backslash\{0\}$, $k\in\mathbb{N}_{0}$ and $L^{n-1}_{k}$ denotes the generalized Laguerre polynomial of degree $k$ and order $n-1$ which is normalized to attain value $1$ when evaluated at $0$, namely $$L^{n-1}_{k}(x)=(n-1)!\sum\limits_{j=0}^{k}\binom{k}{j}\frac{(-x)^{j}}{(j+n-1)!}.$$ Let us notice that there exists other bounded $U(n)$-invariant spherical functions, which do not depend on the variable $t$ and play no role in the subsequent analysis.

\par The family of spherical functions is a powerful tool: it provides us with an expansion formula for  $f\in\mathcal{S}(\HH_{n})$. Indeed, we have that (see for example \cite{T}): \begin{align*}
 f(z,t)=&\sum\limits_{k\ge 0} \int\limits_{-\infty}^{\infty} (f\ast \varphi_{\lambda,k})(z,t) |\lambda|^{n} d\lambda, \quad (z,t)\in\mathbb{H}_{n}. 
\end{align*} Also, each spherical function $\varphi_{\lambda,k}$ is an eigenfunction for both the operators $L$ and $T$ (again, see for example \cite{T}): \begin{align*}
 L\varphi_{\lambda,k} =& -|\lambda|(2k+n)\varphi_{\lambda,k}, \\
 -iT\varphi_{\lambda,k} =& \lambda\varphi_{\lambda,k}. 
 \end{align*} and moreover,
\begin{align*}
-Lf(z,t)=& \sum\limits_{k\ge 0} \int\limits_{-\infty}^{\infty}|\lambda|(2k+n)(f\ast \varphi_{\lambda,k})(z,t) |\lambda|^{n} d\lambda,\\
-iTf(z,t)=& \sum\limits_{k\ge 0} \int\limits_{-\infty}^{\infty}\lambda(f\ast \varphi_{\lambda,k})(z,t) |\lambda|^{n} d\lambda.
\end{align*}

\par From this expansion formulas we are tempted to define an operator $|T|$ as follows: for a function $f$ in, say, the Schwarz class $\mathcal{S}(\mathbb{H}_{n})$ and $(z,t)\in\mathbb{H}_{n}$, let 
\begin{align}\label{|T|}
|T|f(z,t)=\sum\limits_{k\ge 0}\int\limits_{-\infty}^{\infty} |\lambda| (f\ast\varphi_{\lambda,k})(z,t) |\lambda|^{n} d\lambda.
\end{align}this definition is by functional calculation
Thus, we would have that $$|T|\varphi_{\lambda,k} = |\lambda| \varphi_{\lambda,k}.$$ Indeed, this definition  is justified by the functional calculus because the spherical functions determine the spectral decomposition of the operators $\mathcal{L}$ and $T$ (see Theorem 2.4 of \cite{Str}), and formula \eqref{|T|} is valid for defining an operator, which we named $|T|$ (see pag. 181 of \cite{Fo}).

\par Let us consider next a fixed scalar $\alpha\in\RR$ which is different from $2k+n$ for every $k\in\mathbb{N}_{0}$. Then, the operator $\mathcal{L}_{\alpha}:=L+\alpha|T|$ satisfies 
\begin{align*}
\mathcal{L}_{\alpha}\varphi_{\lambda,k}= & (L+\alpha|T|)\varphi_{\lambda,k} \\
 = & (-|\lambda|(2k+n)+\alpha|\lambda|)\varphi_{\lambda,k} \\
= & -|\lambda|(2k+n-\alpha)\varphi_{\lambda,k}.
\end{align*}

\par Since this operator also doesn't belong to the universal enveloping algebra, let us define in this context a fundamental solution for an operator $D$ as a tempered distribution $\Phi$ such that for every $f$ in the Schwartz class, $$D(f\ast \Phi)=(Df)\ast\Phi = f\ast D(\Phi)=f.$$ Hence if we define the operator $K$ by $Kf=f\ast\Phi$, it turns out that $$K\circ Df=D\circ Kf=f.$$

\par As antecedents of finding explicit fundamental solutions let us recall some related bibliography. We begin by citing the inspiring work of Folland \cite{F} who constructed the fundamental solution for the sublaplacian $L$, and then the work of Folland and Stein \cite{FS} who, among other very interesting results, obtain an explicit formula for the fundamental solution for $L+i\alpha T$. Then, Benson, Dooley and Ratcliff study the fundamental solutions for powers of $L$ in \cite{BDR}. Godoy and Saal developed a method for the case when the group $U(p,q)$ acts on $\mathbb{H}_{n}$ instead of the unitary group $U(n)$, they considered the sublaplacian $L$. The case $L+\alpha T$ was obtained in \cite{CS}, by succesfully adapting said method. The general case of the sublaplacian in type H groups was studied first by Kaplan in \cite{K}, then followed by many others. In \cite{GS} was also studied the case of the sublaplacian for the so called quaternionic Heisenberg group, also under the action of a noncompact unitary group.

\par Following these lines, it is our intention to adapt the method from \cite{GS} for our operator $L_{\alpha}$. In this direction we then propose as a candidate to be a fundamental solution (in the sense described above) the distribution $\Phi_{\alpha}$ defined for $f\in\mathcal{S}(\mathfrak{h}_{n})$ as 
\begin{align}\label{Sol.Fund}
<-\Phi_{\alpha},f> = \sum\limits_{k\ge 0} \int\limits_{-\infty}^{\infty} \frac{1}{-|\lambda|(2k+n-\alpha)} <\varphi_{\lambda,k},f> |\lambda|^{n} d\lambda.
\end{align}
The main Theorem of the present work gives us the well definedness of the distribution $\Phi_{\alpha}$ as well as an expression for it in terms of the Gauss hypergeometric function $_{2}F_{1}$: 
\begin{thm}\label{thm:main}
$\Phi_{\alpha}$ is well defined for $f\in\mathcal{S}(\mathbb{H}_{n})$ and
\begin{align*}
<\Phi_{\alpha},f> & = \frac{-4^{n}(n-1)!}{n-\alpha} \\
& \int\limits_{\mathbb{H}^{n}} \mathfrak{Re}\left(\left(\frac{|z|^{2}+4it}{|z|^{4}+16t^{2}}\right)^{n} \,_{2}F_{1}\left(n,\frac{n-\alpha}{2},\frac{n-\alpha}{2}+1;-\frac{(|z|^{2}+4it)^{2}}{|z|^{4}+16t^{2}} \right)\right) \\
& \qquad f(z,t)  dzdt.
\end{align*}
\end{thm}

\par The preliminaries have been given in this section. In section \ref{section:fund.sol} we prove the main theorem. In section \ref{section:alpha.lesser.n} we consider $\alpha<n$ and by means of the Integral Representation Theorem for the Gauss hypergeometric function we are able to give a more detailed expression for $\Phi_{\alpha}$ and finally deduce that for $\alpha=0$ we recover Folland's fundamental solution for the Heisenberg sublaplacian.

\section{The distribution $\Phi_{\alpha}$}\label{section:fund.sol}

\subsection{Well definedness}\label{section:fund.sol.well.def}

\par Let us see that $\Phi_{\alpha}$ is well defined in the sense of distributions. Indeed, we have that for $f\in\mathcal{S}(\mathfrak{h}_{n})$
\begin{align} \nonumber
|<-\Phi_{\alpha},f>|\le & \sum\limits_{k\ge 0} \frac{1}{|2k+n-\alpha|} \int\limits_{-\infty}^{\infty} |<\varphi_{\lambda,k},f>| |\lambda|^{n-1} d\lambda \\ \nonumber 
\le & \sum\limits_{k\ge 0} \frac{1}{|2k+n-\alpha|} \int\limits_{-\infty}^{\infty} \sum\limits_{||\beta||=k} |<E_{\lambda}(h_{\beta},h_{\beta}),f>|  |\lambda|^{n-1} d\lambda \\
= & \sum\limits_{k\ge 0} \frac{1}{|2k+n-\alpha|} \int\limits_{-\infty}^{\infty} \sum\limits_{||\beta||=k} |<\pi_{\lambda}(f)h_{\beta},h_{\beta}>|  |\lambda|^{n-1} d\lambda. \label{well.def.phialpha.1}
\end{align}

In order to see that the expression above is finite, let us recall a few facts about representation theory in general and about the Schr\"{o}dinger representation $\pi_{\lambda}$ in particular (see for example \cite{T}). Each $\pi_{\lambda}:\mathbb{H}_{n}\to\mathcal{U}(L^{2}(\mathbb{R}^{n}))$ is a strongly continuous, irreducible, unitary representation of the Lie group $\mathbb{H}_{n}$ on the Hilbert space $L^{2}(\mathbb{R}^{n})$. Thus, for $f\in L^{1}(\mathbb{H}_{n})\cap L^{2}(\mathbb{H}_{n})$ we have a well defined Hilbert-Schmidt operator $\pi_{\lambda}(f):L^2(\mathbb{R}^{n})\to L^2(\mathbb{R}^{n})$ with Hilbert-Schmidt norm satisfying \begin{align}\label{pilambda.HS.bounded} ||\pi_{\lambda}(f)||_{HS}\le ||f||_{1}.\end{align} A function $\varphi\in L^{2}(\mathbb{R}^{n})$ is a $C^{\infty}$ vector for $\pi_{\lambda}$ if the map from $\mathbb{H}_{n}$ to $L^{2}(\mathbb{R}^{n})$ given by $\pi_{\lambda}(z,t)\varphi$ is $C^{\infty}$. The subspace of $C^{\infty}$ vectors for $\pi_{\lambda}$, called the G\.{a}rding space of $\pi_{\lambda}$, is dense in $L^{2}(\mathbb{R}^{n})$ and it coincides with the Schwarz class $\mathcal{S}(\mathbb{R}^{n})$. Also, for an operator $D\in\mathcal{U}(\mathfrak{H}_{n})$ we have a well defined operator $d\pi_{\lambda}(D)$  defined in the G\.{a}rding space called the infinitesimal representation. For the left invariant vector fields conforming the canonical basis of $\mathfrak{h}_{n}$ we have that if $\varphi\in\mathcal{S}(\mathbb{R}^{n})$ and $\xi\in\mathbb{R}^{n}$,
\begin{align*}
d\pi_{\lambda}(X_{j})\varphi(\xi)=& \frac{\partial}{\partial \xi_{j}}\varphi(\xi),\qquad j=1,\dots,n, \\
d\pi_{\lambda}(Y_{j})\varphi(\xi)=& i \lambda \xi_{j} \varphi(\xi),\qquad j=1,\dots,n, \\
d\pi_{\lambda}(T)\varphi(\xi)=& i \lambda \varphi(\xi).
\end{align*}
From this formulas we quickly get that for the sublaplacian we have \begin{align*} d\pi_{\lambda}(L)\varphi(\xi)=&-\triangle \varphi(\xi)+\lambda^{2}|\xi|^{2}\varphi(\xi), \end{align*} which is exactly the scaled Hermite operator $H(\lambda)=-\triangle+\lambda^{2}|\xi|^{2}$, whose eigenfunctions are non other than the scaled Hermite functions $h_{\beta}$. Finally, a computation gives us \begin{align*} \pi_{\lambda}(Lf)h_{\beta}=& \pi_{\lambda}(f)d\pi_{\lambda}(L)h_{\beta}=\pi_{\lambda}(f)(-|\lambda|(2||\beta||+n))h_\beta,\end{align*}
that is, 
\begin{align*} 
\pi_{\lambda}(f)h_{\beta}=& -\frac{1}{|\lambda|(2||\beta||+n)}\pi_{\lambda}(Lf) h_{\beta},  
\end{align*}
and by recursively applying this formula we get for any $m\in\mathbb{N}$, 
\begin{align}\label{m.recursively}
\pi_{\lambda}(f)h_{\beta}=& (-1)^{m}\frac{1}{(|\lambda|(2||\beta||+n))^{m}}\pi_{\lambda}(L^{m}f) h_{\beta}.
\end{align}
We will use this fact in order to bound expression \eqref{well.def.phialpha.1}. Indeed, from all of the above and since the special Hermite functions are orthonormal, it is clear that 
\begin{align*}
|<\pi_{\lambda}(f)h_{\beta},h_{\beta}>| \le & ||\pi_{\lambda}(f)||_{HS} \le ||f||_{1},
\end{align*}
the last inequality following from \eqref{pilambda.HS.bounded}. Then we know from elementary combinatorial analysis that the sum for every multiindex $\beta=(\beta_{1},\dots,\beta_{n})\in\mathbb{N}_{0}$ such that $||\beta||=\beta_{1}+\dots+\beta_{n}=k$ of a constant independent of $\beta$ equals $\binom{k+n-1}{n-1}$ times the constant.

\par Now we are in position to see that expression \eqref{well.def.phialpha.1} is finite. Indeed, by writing $\mathbb{R}$ as the disjoint union of the sets $A=\{\lambda\in\mathbb{R}: 0\le |\lambda| |2k+n-\alpha| <1\}$ and $B=\{\lambda\in\mathbb{R}: |\lambda| |2k+n-\alpha| \ge 1\}$, we may split the integral into two summands and then use \eqref{m.recursively} as follows:
\begin{align*}
|<-\Phi_{\alpha},f>|\le & \sum\limits_{k\ge 0} \frac{1}{|2k+n-\alpha|} \int\limits_{A} \sum\limits_{||\beta||=k} |<\pi_{\lambda}(f)h_{\beta},h_{\beta}>|  |\lambda|^{n-1} d\lambda \\
 & + \sum\limits_{k\ge 0} \frac{1}{|2k+n-\alpha|} \int\limits_{B} \sum\limits_{||\beta||=k} |<\pi_{\lambda}(f)h_{\beta},h_{\beta}>|  |\lambda|^{n-1} d\lambda \\
= & I + II,
\end{align*}
hence we study each one separately. First we apply \eqref{m.recursively}, and then the key is to change the variable of integration according to $\sigma=\lambda|2k+n-\alpha|$, in both cases:
\begin{align*}
I\le & ||L^{m}f||_{1} \sum\limits_{k\ge 0} \binom{k+n-1}{n-1} \frac{1}{|2k+n-\alpha|} \frac{1}{|2k+n|^{m}} \int\limits_{A}  |\lambda|^{n-1-m} d\lambda \\
= & ||L^{m}f||_{1} \sum\limits_{k\ge 0} \binom{k+n-1}{n-1} \frac{1}{|2k+n-\alpha|^{n+1-m}} \frac{1}{|2k+n|^{m}} \int\limits_{0\le |\sigma| < 1}  |\sigma|^{n-1-m} d\sigma,
\end{align*}
and since the integral is finite for $m<n$ and $\binom{k+n-1}{n-1} \sim k^{n-1}$, we have that the summands are of order $\frac{1}{k^{2}}$, hence convergent.

Now, for the second one
\begin{align*}
II\le & ||L^{m}f||_{1} \sum\limits_{k\ge 0} \binom{k+n-1}{n-1} \frac{1}{|2k+n-\alpha|} \frac{1}{|2k+n|^{m}} \int\limits_{B}  |\lambda|^{n-1-m} d\lambda \\
= & ||L^{m}f||_{1} \sum\limits_{k\ge 0} \binom{k+n-1}{n-1} \frac{1}{|2k+n-\alpha|^{n+1-m}} \frac{1}{|2k+n|^{m}} \int\limits_{|\sigma|\ge 1}  |\sigma|^{n-1-m} d\sigma,
\end{align*}
and since the integral is finite for $m>n$ and $\binom{k+n-1}{n-1} \sim k^{n-1}$, we have that the summands are of order $\frac{1}{k^{2}}$, hence convergent.

\subsection{Explicit formula for the fundamental solution}\label{subsection:explicit}

\par In order to explicitly compute $\Phi_{\alpha}$ we begin by plugging in the explicit formula for the spherical functions \eqref{spherical.functions.explicit} into \eqref{Sol.Fund}. Observe first that, since we proved $\Phi_{\alpha}$ is well defined in the sense of distributions, we are allowed to interchange the integrals in $\mathbb{R}$ and $\mathbb{H}_{n}$. 
 Thus, formally, for a function $f\in\mathcal{S}(\mathbb{H}_{n})$,
\begin{align*}
<-\Phi_{\alpha},f> = &   \sum\limits_{k\ge 0} \frac{1}{2k+n-\alpha} \\
&  <\int\limits_{-\infty}^{\infty}e^{i\lambda t}L_{k}^{n-1}\left( {\frac{{|\lambda|}}{2}} |z|^{2} \right) e^{-\frac{|\lambda|}{4}|z|^{2}}|\lambda|^{n-1} d\lambda,f>\\
= &   \sum\limits_{k\ge 0} \frac{1}{2k+n-\alpha} \\
&  \int\limits_{\mathbb{H}_{n}}\int\limits_{-\infty}^{\infty}e^{i\lambda t}L_{k}^{n-1}\left( {\frac{{|\lambda|}}{2}} |z|^{2} \right) e^{-\frac{|\lambda|}{4}|z|^{2}}|\lambda|^{n-1} d\lambda f(z,t)  dzdt,
\end{align*}
then from Abel's Lemma we can write
\begin{align*}
<-\Phi_{\alpha},f> = & \lim\limits_{r\to 1^{-}} \lim\limits_{\epsilon\to 0^{+}} \sum\limits_{k\ge 0} \frac{r^{2k+n-\alpha}}{2k+n-\alpha} \\
&   \int\limits_{\mathbb{H}_{n}}\ \int\limits_{-\infty}^{\infty} e^{-\epsilon|\lambda|} e^{i\lambda t}L_{k}^{n-1}\left( \frac{|\lambda|}{2} |z|^{2} \right) e^{-\frac{|\lambda|}{4}|z|^{2}}|\lambda|^{n-1}  d\lambda  f(z,t) dzdt,
\end{align*}
and the generating identity for the Laguerre polynomials allows us to compute the integral on $\lambda$ (see formulas (4.8) and (4.9) of \cite{GS2} for a detailed proof):
\begin{align}\label{nosemeocurre}
& \int\limits_{-\infty}^{\infty} e^{-\epsilon|\lambda|} e^{i\lambda t}L_{k}^{n-1}\left( \frac{|\lambda|}{2} |z|^{2} \right) e^{-{\frac{|\lambda|}{4}}|z|^{2}}|\lambda|^{n-1} d\lambda  = 2 \beta_{n} \alpha_{k} \mathfrak{Re}  \left( {\frac{(|z|^{2}+4it+4\epsilon)^{k}}{(|z|^{2}-4it+4\epsilon)^{n+k}}} \right),
\end{align}
where
\begin{align}\label{beta.n}
\beta_{n} & = 4^{n}(n-1)!, \\ \label{alpha.k}
\alpha_{k} & = (-1)^{k} \binom{k+n-1}{n-1}.
\end{align}
The first step will be to take limit when $\epsilon$ goes to $0$, and with this in mind we observe that
\begin{align}\label{mod1.fs}
&\frac{(|z|^{2}+4it+4\epsilon)^{k}}{(|z|^{2}-4it+4\epsilon)^{n+k}} = \left(\frac{|z|^{2}+4it+4\epsilon}{|z|^{2}-4it+4\epsilon}\right)^{k+\frac{n}{2}} \frac{1}{[(|z|^{2}+4\epsilon)^{2}+16t^{2}]^{\frac{n}{2}	}},
\end{align}
and since the first factor on the right hand side of the above formula is a complex number of modulus $1$, we have that for any $\epsilon >0$
\begin{align}\label{kernel.bound}
&\left| \left(\frac{|z|^{2}+4it+4\epsilon}{|z|^{2}-4it+4\epsilon}\right)^{k+\frac{n}{2}} \frac{1}{[(|z|^{2}+4\epsilon)^{2}+16t^{2}]^{\frac{n}{2}}} \right| \le  \frac{1}{(|z|^{4}+16t^{2})^{\frac{n}{2}}}. 
\end{align}
Then, 
\begin{align*}
<-\Phi_{\alpha},f> = & \lim\limits_{r\to 1^{-}} \lim\limits_{\epsilon\to 0^{+}} \sum\limits_{k\ge 0} \frac{r^{2k+n-\alpha}}{2k+n-\alpha} \\
&   \int\limits_{\mathbb{H}_{n}} \mathfrak{Re} \left( \left(\frac{|z|^{2}+4it+4\epsilon}{|z|^{2}-4it+4\epsilon}\right)^{k+\frac{n}{2}} \right) \frac{1}{[(|z|^{2}+4\epsilon)^{2}+16t^{2}]^{\frac{n}{2}}}  f(z,t) dzdt,
\end{align*}
and in order to interchange limit with integral (that is, to apply Lebesgue's Dominated Convergence Theorem), from \eqref{kernel.bound} we only need to prove that 
\begin{align}\label{finitude}
I = & \int\limits_{\mathbb{H}_{n}} \frac{1}{(|z|^{4}+16t^{2})^{\frac{n}{2}}}  f(z,t) dzdt \qquad < \infty.
\end{align}
 First let us write $\mathbb{H}_{n}=\mathbb{C}^{n}\times\mathbb{R}$ and consider spherical coordinates in $\mathbb{C}^{n}$ as follows: if $(z,t)\in\mathbb{H}_{n}$, then $z=\xi\cdot R\in S^{2n-1}\times\mathbb{R}_{0}^{+}$, where obviously $|z|=R$ and the Jacobian gives us $dz=R^{2n-1}dRd\sigma(\xi)$, with $d\sigma(\xi)$ denoting the volume element:
\begin{align*}
I  = & \int\limits_{\mathbb{R}} \int\limits_{S^{2n-1}} \int\limits_{\mathbb{R}_{0}^{+}}  \frac{1}{(R^{4}+16t^{2})^{\frac{n}{2}}}  f(\xi\cdot R,t) R^{2n-1}dRd\sigma(\xi)dt.
\end{align*}
Now let us consider the averaging function $A$ defined for $f\in\mathcal{S}(\mathbb{H}_{n})$ by \begin{align}\label{Af} Af(|z|,t)=& \frac{1}{|S^{2n-1}|} \int\limits_{S^{2n-1}} f(\xi\cdot |z|,t) d\sigma(\xi). \end{align} Obviously $Af$ is a radial function in the sense that $Af(|z_{1}|,t)=Af(|z_{2}|,t)$ if $|z_{1}|=|z_{2}|$. A new change of variables according to $R^{2}=\tau$ gives us
\begin{align*}
I = &  \frac{|S^{2n-1}|}{2} \int\limits_{\mathbb{R}} \int\limits_{\mathbb{R}_{0}^{+}}   \frac{1}{(\tau^{2}+16t^{2})^{\frac{n}{2}}}  Af(\tau^{\frac{1}{2}},t) \tau^{n-1}d\tau dt.
\end{align*}
With these computations, the following Lemma justifies \eqref{finitude}.

\begin{lem}\label{integrability.in.halfplane}
The function
\begin{align*}
g(\tau,t) =& \frac{1}{(\tau^{2}+16t^{2})^{\frac{n}{2}}}  Af(\tau^{\frac{1}{2}},t) \tau^{n-1}
\end{align*}
is integrable on the halfplane
\begin{align*}
\mathcal{H} = & \{(\tau,t)\in\mathbb{R}^{2}:\tau >0\}.
\end{align*}
\end{lem}
\begin{proof}
We will proceed by dividing $\mathcal{H}$ into three regions. First we define the half ellipse region
\begin{align*}
B = & \{(\tau,t)\in \mathcal{H} : \tau^{2}+16t^{2} < 1\}.
\end{align*}
This region plays the role of the unitary ball in $\mathcal{H}$ with respect to the metric distance in $\mathbb{H}_{n}$ (recall \eqref{metric}). Indeed, if $(\tau,t)\in B$ then $|(\xi\cdot\tau^{\frac{1}{2}}, t)| = | (\tau^{\frac{1}{2}})^{4}+16t^{2} | < 1$. Hence, for $(\tau,t)\in B$,
\begin{align*}
|g(\tau,t)| \le & \frac{1}{(\tau^{2}+16t^{2})^{\frac{n}{2}}}  \frac{1}{|S^{2n-1}|}\int\limits_{S^{2n-1}} |f(\xi\cdot\tau^{\frac{1}{2}},t)| |d\sigma(\xi)| \tau^{n-1} \\
 \le & c \frac{\tau^{n-1}}{(\tau^{2}+16t^{2})^{\frac{n}{2}}},
\end{align*}
and the function on the right hand side of the inequality is obviously integrable in $B$.
Next, let us now define the strip
\begin{align*}
S = & \left\{(\tau,t)\in \mathcal{H}\backslash B : |t| < \frac{1}{8}\right\}.
\end{align*}
If $(\tau,t)\in S$, it is elemental to see that $\tau>\frac{1}{2}$,
\begin{align*}
\int\limits_{S} |g(\tau,t)|d\tau dt \le & \int\limits_{|t|< \frac{1}{8}} \int\limits_{\tau>\frac{1}{2}} \frac{1}{(\tau^{2}+16t^{2})^{\frac{n}{2}}}  |Af(\tau^{\frac{1}{2}},t)|\tau^{n-1} d\tau dt \\
 \le & \int\limits_{|t|< \frac{1}{8}} \int\limits_{\tau>\frac{1}{2}} \frac{1}{\tau} \frac{1}{|S^{2n-1}|} \int\limits_{S^{2n-1}} |f(\xi\cdot\tau^{\frac{1}{2}},t)| |d\sigma(\xi)| d\tau dt \\
\le & \int\limits_{|t|< \frac{1}{8}} \int\limits_{\tau>\frac{1}{2}} \frac{c}{\tau^{2}} d\tau dt \le c, 
\end{align*}
where the second inequality follows from the fact that $\tau^{2}+16t^{2}>\tau^{2}$, the third from the fact that $f\in\mathcal{S}(\HH_{n})$.
The third region is then the complement in $\mathcal{H}$ of $B\cup S$. In this region one has that $\frac{\tau^{n-1}}{(\tau^{2}+16t^{2})^{\frac{n}{2}}} \le \frac{\tau^{n-1}}{(\tau^{2}+\frac{1}{4})^{\frac{n}{2}}}$, and again since $f$ is a Schwartz function, its average is bounded by, for example, $\frac{c}{\tau t^{2}}$. Thus, by similar arguments we can assure that the integral of $|g(\tau,t)|$ in this region is also finite.
\end{proof}

\par Applying Lebesgue's Dominated Convergence Theorem it follows that

\begin{prop}
\begin{align*}
\lim\limits_{\epsilon\to 0^{+}}  \int\limits_{\HH_{n}} \mathfrak{Re}     \left( {\frac{(|z|^{2}+4it+4\epsilon)^{k}}{(|z|^{2}-4it+4\epsilon)^{n+k}}} \right)  f(z,t) dzdt = \int\limits_{\HH_{n}} \mathfrak{Re}     \left( {\frac{(|z|^{2}+4it)^{k}}{(|z|^{2}-4it)^{n+k}}} \right)  f(z,t) dzdt. 
\end{align*}
\end{prop}

\par Using again formula \eqref{mod1.fs} we are able to write
\begin{align*}
<-\Phi_{\alpha},f> = & 2 \beta_{n}  \lim\limits_{r\to 1^{-}} \sum\limits_{k\ge 0} \alpha_{k} \frac{r^{2k+n-\alpha}}{2k+n-\alpha} \\
&   \int\limits_{\mathbb{H}_{n}}  \mathfrak{Re} \left( \left(\frac{|z|^{2}+4it}{|z|^{2}-4it}\right)^{k+\frac{n}{2}} \right) \frac{1}{(|z|^{4}+16t^{2})^{\frac{n}{2}}}  f(z,t) dzdt,
\end{align*}
and by performing the exact same variable changes as we did just above of Lemma \ref{integrability.in.halfplane} we obtain that 
\begin{align*}
<-\Phi_{\alpha},f> = & 2 \beta_{n}  \lim\limits_{r\to 1^{-}} \sum\limits_{k\ge 0} \alpha_{k} \frac{r^{2k+n-\alpha}}{2k+n-\alpha} \\
& \frac{|S^{2n-1}|}{2} \int\limits_{\mathbb{R}} \int\limits_{\mathbb{R}_{0}^{+}}  \mathfrak{Re}\left(\left(\frac{\tau+4it}{\tau-4it}\right)^{k+\frac{n}{2}}\right)  \frac{1}{(\tau^{2}+16t^{2})^{\frac{n}{2}}}  Af(\tau^{\frac{1}{2}},t) \tau^{n-1}d\tau dt.
\end{align*}
In order to integrate in the halfplane we need to change variables one more time: let $\tau+4it=\rho e^{i\theta}$, then
\begin{align}\nonumber
<-\Phi_{\alpha},f>=& \hat{\beta_{n}} \lim\limits_{r\to 1^{-}} \sum\limits_{k\ge 0}\alpha_{k}\frac{r^{2k+n-\alpha}}{2k+n-\alpha}  \\ \nonumber
& \int\limits_{-\frac{\pi}{2}}^{\frac{\pi}{2}}  \mathfrak{Re}(e^{i(2k+n)\theta}) \cos^{n-1}\theta \int\limits_{0}^{\infty} Af(\rho^{\frac{1}{2}}\cos^{\frac{1}{2}}\theta,\frac{\rho}{4}\sin\theta)d\rho d\theta \\ \label{Phi.r.alpha.resumida}
= & \hat{\beta_{n}} \lim\limits_{r\to 1^{-}} \int\limits_{-\frac{\pi}{2}}^{\frac{\pi}{2}} (\mathfrak{Re}\Psi_{r,\alpha}(\theta)) K_{f}(\theta) d\theta,
\end{align}
where 
\begin{align} \nonumber
\hat{\beta_{n}} =& \frac{|S^{2n-1}|}{8} \beta_{n}, \\
\Psi_{r,\alpha}(\theta) =& \sum\limits_{k\ge 0}\alpha_{k}\frac{r^{2k+n-\alpha}}{2k+n-\alpha} e^{i(2k+n)\theta},\\
K_{f}(\theta) =& \cos^{n-1}\theta \int\limits_{0}^{\infty} Af(\rho^{\frac{1}{2}}\cos^{\frac{1}{2}}\theta,\frac{\rho}{4}\sin\theta)d\rho.
\end{align}

In order to compute the limit when $r$ goes to $1$ from below, we need to give sense to the expression $\Psi_{r,\alpha}\to \Psi_{1,\alpha}$. With this in mind, let us define
\begin{align}\label{Psi.alpha}
\Psi_{\alpha}(\theta) := \Psi_{1,\alpha} =& \sum\limits_{k\ge 0}\alpha_{k}\frac{1}{2k+n-\alpha} e^{i(2k+n)\theta},
\end{align}
and let us consider the vector space $\mathcal{X}$ of those functions supported in the half circle $\left[ -\frac{\pi}{2},\frac{\pi}{2} \right]$ which have $n-2$ continuous derivatives vanishing in $\pm\frac{\pi}{2}$ and also have essentially bounded $n-1$ derivative. Explicitly,

\begin{align}\label{X}
\mathcal{X}=& \left\{ g\in C^{n-2}\left(\left[ -\frac{\pi}{2},\frac{\pi}{2} \right]\right) : g^{(j)}\left( \pm\frac{\pi}{2} \right)=0, 0\le j \le n-2, g^{(n-1)}\in L^{\infty} \left(\left[ -\frac{\pi}{2},\frac{\pi}{2} \right]\right) \right\}.
\end{align}
Thus defined, the function $K_{f}$ belongs to the space $\mathcal{X}$ (just observe that each one of its derivatives has a power of the cosine function as a factor). We shall identify $g\in\mathcal{X}$ with a function $\tilde{g}$ defined by $0$ outside $supp(g)$, and make no distinction between $g$ and $\tilde{g}$. Hence we can regard $g\in\mathcal{X}$ as a function belonging to $C^{n-2}(S^{1})$ with $g^{n-1}\in L^{\infty}(S^{1})$ supported in $\left[-\frac{\pi}{2},\frac{\pi}{2}\right]$. Also, observe that if $g\in\mathcal{X}$, then $e^{i\alpha\theta}g\in\mathcal{X}$. Finally, let us consider the topology in $\mathcal{X}$ given for $g\in\mathcal{X}$ by
\begin{align*}
||g||_{\mathcal{X}} & = \max\{ ||g^{(j)}||_{\infty}, 0\le j \le n-1 \},
\end{align*}
which makes $\mathcal{X}$ a Banach space. Thus, if $g\in\mathcal{X}$, $<\Psi_{\alpha},g>$ is well defined:
\begin{align*}
|<\Psi_{\alpha},g>| \le &  \sum\limits_{k\ge 0} \binom{k+n-1}{k} \frac{|<e^{i(2k+n)\theta},g>|}{|2k+n-\alpha|},
\end{align*}
and if we denote the $m$-th Fourier coefficient of $g$ by $\hat{g}(m)=<e^{im\theta},g>$, we have that
\begin{align*}
|<\Psi_{\alpha},g>| \le & c \sum\limits_{k\ge 0} \frac{k^{n-1}}{|2k+n|^{n-1}} \frac{|\hat{g}^{(n-1)}(2k+n)|}{|2k+n-\alpha|} \\
\le & c \sum\limits_{k\ge 0} \frac{1}{k} |\hat{g}^{(n-1)}(2k+n)| \le c \left(\sum\limits_{k\ge 0} \frac{1}{k^{2}} \right)^{2}||g^{(n-1)}||_{L^{2}}\\
&\leq c||g||_{\mathcal{X}},
\end{align*}
where each $c$ denotes a constant not necessarily the same on every inequality, and the last inequality is just the Cauchy-Schwarz inequality.
Hence $\Psi_{\alpha}\in\mathcal{X}'$.
 Furthermore, Abel's Lemma gives us $$\lim\limits_{r\to 1^{-}} \Psi_{r,\alpha} = \Psi_{\alpha}$$ in $\mathcal{X}'$, with the weak convergence topology.

Let us now compute $<\Psi_{\alpha},g>$ for $g\in\mathcal{X}$. We can rewrite the series defining $\Psi_{r,\alpha}$ and $\Psi_{\alpha}$ to get the Gauss Hypergeometric Function $_{2}F_{1}$ which we recall next (see, for example, the classical book Slater \cite{Sl})\footnote{See also the book of Rainville on Special Functions \cite{R} and the comprehensive work of \cite{H}.}: for parameters $a,b,c\in\mathbb{C}$, $c\notin\mathbb{Z}^{-}$ the Gauss Hypergeometric Function is defined as
\begin{align*}
_{2}F_{1}(a,b,c;\zeta)=& \sum\limits_{k=0}^{\infty} \frac{(a)_{k}(b)_{k}}{(c)_{k}} \frac{\zeta^{k}}{k!},
\end{align*}
for $|\zeta|<1$, and by analytic continuation elsewhere, where $(a)_{k}=a(a+1)(a+2)\dots(a+k-1)$ if $k\in\mathbb{N}$ and $(a)_{k}=1$ for $k= 0$ and $a\neq 0$ is the Pochhammer symbol, which can be also written in terms of the well known Gamma Function $\Gamma$: if $a\notin\mathbb{Z}^{-}$
\begin{align}\label{Pochammer}
(a)_{k}=&\frac{\Gamma(a+k)}{\Gamma(a)}.
\end{align} With this in mind, together with the factorial property of the Gamma function, namely $\Gamma(a+1)=a\Gamma(a)$, we have
\begin{align} \nonumber
\Psi_{r,\alpha}(\theta) = & r^{n-\alpha}e^{in\theta}\sum\limits_{k=0}^{\infty}\frac{(k+n-1)!}{(n-1)!} \frac{1}{2\left(k+\frac{n-\alpha}{2}\right)} \frac{(-r^{2}e^{2i\theta})^{k}}{k!} \\ \nonumber
=& \frac{r^{n-\alpha}e^{in\theta}}{2}\sum\limits_{k=0}^{\infty}(n)_{k} \frac{\Gamma\left(k+\frac{n-\alpha}{2}\right)}{\Gamma\left(k+\frac{n-\alpha}{2}+1\right)}\frac{(-r^{2}e^{2i\theta})^{k}}{k!} \\ \nonumber
=& \frac{r^{n-\alpha}e^{in\theta}}{n-\alpha}\sum\limits_{k=0}^{\infty} \frac{(n)_{k}\left(\frac{n-\alpha}{2}\right)_{k}}{\left(\frac{n-\alpha}{2}+1\right)_{k}}\frac{(-r^{2}e^{2i\theta})^{k}}{k!} \\ \label{Phi.r.alpha.hypergeom}
=& \frac{r^{n-\alpha}e^{in\theta}}{n-\alpha} \,_{2}F_{1}\left(n,\frac{n-\alpha}{2},\frac{n-\alpha}{2}+1;-r^{2}e^{2i\theta}\right).
\end{align}
The weak convergence gives us, when $r\to 1^{-}$, that for a function $g\in\mathcal{X}$,
\begin{align} \label{Phi.alpha.hypergeom}
<\Psi_{\alpha}(\theta),g> =& <\frac{e^{in\theta}}{n-\alpha} \,_{2}F_{1}\left(n,\frac{n-\alpha}{2},\frac{n-\alpha}{2}+1;-e^{2i\theta}\right),g>.
\end{align}

For the sake of notation, let us call the Gauss hypergeometric function with parameters $n,\frac{n-\alpha}{2},\frac{n-\alpha}{2}+1$ as follows:
\begin{align}\label{function.F.alpha}
F_{\alpha}(\omega)=& \,_{2}F_{1}\left(n,\frac{n-\alpha}{2},\frac{n-\alpha}{2}+1;\omega \right),
\end{align}
for the appropriate parameter $\omega$.

\par The next step is to go backwards the changes of variables. Let us take over from \eqref{Phi.r.alpha.resumida}, take the weak limit when $r\to 1^{-}$ and apply formula \eqref{Phi.alpha.hypergeom}:
\begin{align*}
<-\Phi_{\alpha},f> = & \frac{4^{n-1}(n-1)!}{2(n-\alpha)} |S^{2n-1}| \int\limits_{-\frac{\pi}{2}}^{\frac{\pi}{2}} \mathfrak{Re}\left( e^{in\theta} F_{\alpha}\left(-e^{2i\theta}\right) \right)  \\
& \times \cos^{n-1}\theta \int\limits_{0}^{\infty} Af(\rho^{\frac{1}{2}}\cos^{\frac{1}{2}}\theta,\frac{\rho}{4}\sin\theta)d\rho d\theta \\
= & \frac{4^{n}(n-1)!}{2(n-\alpha)} |S^{2n-1}| \int\limits_{\mathbb{R}} \int\limits_{\mathbb{R}_{0}^{+}} \mathfrak{Re}\left(\frac{(\tau+4it)^{n}}{(\tau^{2}+16t^{2})^{\frac{n}{2}}} F_{\alpha}\left(-\frac{(\tau+4it)^{2}}{\tau^{2}+16t^{2}} \right) \right) \\
& \times  \frac{\tau^{n-1}}{(\tau^{2}+16t^{2})^{\frac{n}{2}}} Af(\tau^{\frac{1}{2}},t) d\tau dt \\
= & \frac{4^{n}(n-1)!}{n-\alpha} |S^{2n-1}| \int\limits_{\mathbb{R}} \int\limits_{\mathbb{R}_{0}^{+}} \mathfrak{Re}\left(\frac{(R^{2}+4it)^{n}}{(R^{4}+16t^{2})^{\frac{n}{2}}} F_{\alpha}\left(-\frac{(R^{2}+4it)^{2}}{R^{4}+16t^{2}} \right) \right) \\
& \times  \frac{R^{2n-1}}{(R^{4}+16t^{2})^{\frac{n}{2}}}  Af(R,t) dRdt \\ 
= & \frac{4^{n}(n-1)!}{n-\alpha} \int\limits_{\mathbb{R}} \int\limits_{S^{2n-1}} \int\limits_{\mathbb{R}_{0}^{+}} \mathfrak{Re}\left(\frac{(R^{2}+4it)^{n}}{(R^{4}+16t^{2})^{\frac{n}{2}}} F_{\alpha}\left(-\frac{(R^{2}+4it)^{2}}{R^{4}+16t^{2}} \right) \right) \\
& \times  \frac{R^{2n-1}}{(R^{4}+16t^{2})^{\frac{n}{2}}}  f(\xi\cdot R,t)  dRd\sigma(\xi)dt \\ 
= & \frac{4^{n}(n-1)!}{n-\alpha} \int\limits_{\mathbb{R}} \int\limits_{\mathbb{C}^{n}} \mathfrak{Re}\left(\left(\frac{|z|^{2}+4it}{|z|^{4}+16t^{2}}\right)^{n} F_{\alpha}\left(-\frac{(|z|^{2}+4it)^{2}}{|z|^{4}+16t^{2}} \right) \right) f(z,t)  dzdt.
\end{align*}
We have thus proved Theorem \ref{thm:main}.

\section{The case $\alpha < n$}\label{section:alpha.lesser.n}

Let us begin by recalling an Integral Representation for the Gauss hypergeometric function (see for example \cite{Sl}). For $|\omega|<1$, $\mathfrak{Re}(c)>0$, $\mathfrak{Re}(b)>0$,
\begin{align*}
_{2}F_{1}(a,b,c;\omega)=&\frac{\Gamma(c)}{\Gamma(b)\Gamma(c-b)}\int\limits_{0}^{1} \frac{{s}^{b-1}(1-s)^{c-b-1}}{(1-\omega s)^{a}}ds.
\end{align*}

Along this section we will consider $\alpha<n$. Then, from equation \eqref{Phi.r.alpha.hypergeom}, since $-r^{2}e^{2i\theta}$ is a complex number of modulus lesser than 1, we are allowed to write
\begin{align*} 
\Psi_{r,\alpha}(\theta) = & \frac{r^{n-\alpha}e^{in\theta}}{2} \int\limits_{0}^{1} \frac{s^{\frac{n-\alpha}{2}-1}}{(1+r^{2}e^{2i\theta}s)^{n}} ds,
\end{align*}
hence the weak convergence also gives us, when $r\to 1^{-}$, that for a function $g\in\mathcal{X}$,
\begin{align*}
<\Psi_{\alpha}(\theta),g> = & \frac{1}{2} <e^{in\theta} \int\limits_{0}^{1} \frac{s^{\frac{n-\alpha}{2}-1}}{(1+e^{2i\theta}s)^{n}} ds , g>,
\end{align*}

In order to compute the real part of this distribution in the weak sense, that is, applied to a function $g$, we need to understand the integral factor. First let us change variables according to $s=u^{2}$ 
\begin{align}\label{int.3}
I= & \int\limits_{0}^{1} \frac{s^{\frac{n-\alpha}{2}-1}}{(1+e^{2i\theta}s)^{n}} ds = 2\int\limits_{0}^{1} \frac{u^{n-\alpha-1}}{(1+e^{2i\theta} u^{2})^{n}} du.
\end{align}
Let us now consider the complex function 
\begin{align*}
h(\zeta) = & \frac{\zeta^{n-\alpha-1}}{(1+\zeta^{2})^{n}}.
\end{align*}
This function has poles in $\zeta = i$ and $\zeta =-i$, and we can apply Cauchy's Theorem for the curve given by the boundary of a circular sector centred at the origin, with radius $1$ and sides on the $x$-axis and on the straight line through the origin forming an angle of width $\theta\in(-\frac{\pi}{2},\frac{\pi}{2})$ with it. That is to say,
\begin{align*}
\int\limits_{\gamma_{3}}h(\zeta)d\zeta = & \int\limits_{\gamma_{1}}h(\zeta)d\zeta + \int\limits_{\gamma_{2}}h(\zeta)d\zeta,
\end{align*}
where $\gamma_{1,2,3}$ are the curves parametrized as follows:
\begin{align*}
\gamma_{1}(x) = & x, \mbox{ for } 0\le x \le 1, \\
\gamma_{2}(\sigma) = & e^{i\sigma}, \mbox{ for } 0\le\sigma\le\theta, \\
\gamma_{3}(u) = & e^{i\theta}u, \mbox{ for } 0\le u\le 1.
\end{align*}
It is straightforward to see that $I=2 e^{-i\theta(n-\alpha)} (I_{1}+I_{2})$, where $I_{j}$ denotes the integral of $h$ on the curve $\gamma_{j}$ ($j=1,2,3$). So we need to compute both $I_{1}$ and $I_{2}$.
\begin{align*}
I_{1} = & \int\limits_{0}^{1} \frac{x^{n-\alpha-1}}{(1+x^{2})^{n}}dx =  \frac{1}{2} \int\limits_{0}^{1} \frac{y^{\frac{n-\alpha}{2}-1}}{(1+y)^{n}} dy  \\
= & \frac{1}{2} \int\limits_{0}^{\frac{1}{2}} \nu^{\frac{n-\alpha}{2}-1} (1-\nu)^{\frac{n+\alpha}{2}-1} d\nu,
\end{align*}
where we changed variables according to $y=x^{2}$ first and $\nu=\frac{y}{y+1}$ second. This last integral is real and it is known as the incomplete Beta function (see for example \cite{D}). Thus,
\begin{align*}
I_{1} = & \frac{1}{2} B_{\frac{1}{2}}\left( \frac{n-\alpha}{2}, \frac{n+\alpha}{2} \right).
\end{align*}

Next let us compute $I_{2}$.
\begin{align*}
I_{2} = & i \int\limits_{0}^{\theta} \frac{e^{i(n-\alpha)\sigma}}{(1+e^{2i\sigma})^{n}} d\sigma = i\int\limits_{0}^{\theta} \frac{e^{i(n-\alpha)\sigma}}{e^{in\sigma}(e^{-i\sigma}+e^{i\sigma})^{n}} d\sigma = \\
= & \frac{i}{2^{n}} \int\limits_{0}^{\theta} \frac{e^{-i\alpha\sigma}}{\cos^{n}\sigma} d\sigma = \frac{1}{2^{n}} \int\limits_{0}^{\theta} \frac{\sin\alpha\sigma}{\cos^{n}\sigma} d\sigma + \frac{i}{2^{n}} \int\limits_{0}^{\theta} \frac{\cos\alpha\sigma}{\cos^{n}\sigma} d\sigma.
\end{align*}

Finally, for a function $g\in\mathcal{X}$,
\begin{align*}
<\mathfrak{Re} \Psi_{\alpha} (\theta) , g> = & <\mathfrak{Re} e^{i\alpha\theta}(I_{1}+I_{2}),g> = \\
 = & <\frac{1}{2}\cos\alpha\theta B_{\frac{1}{2}}\left( \frac{n-\alpha}{2},\frac{n+\alpha}{2} \right) + \frac{1}{2^{n}} \int\limits_{0}^{\theta}\frac{\sin(\alpha(\sigma-\theta))}{\cos^{n}\sigma} d\sigma, g>,
\end{align*}
where the function $$m_{\alpha}(\theta)=\int\limits_{0}^{\theta}\frac{\sin(\alpha(\sigma-\theta))}{\cos^{n}\sigma} d\sigma$$ is the solution to the initial value problem for the second order ordinary differential equation which models the undamped forced motion, namely $y''(\theta)+\alpha^{2}y(\theta)+\alpha\sec^{n}(\theta)=0$, $y(0) = 0$, $y'(0) = 0$. Observe that $m_{0}(\theta)\equiv 0$.

From the above, by reversing the changes of variable, we have that
\begin{align*}
<-\Phi_{\alpha},f> & =  4^{n}(n-1)! \\
& \int\limits_{\mathbb{H}_{n}} \left[ \frac{1}{2}B_{\frac{1}{2}}\left(\frac{n-\alpha}{2},\frac{n+\alpha}{2}\right)\cos\left(\alpha\arctan\frac{4t}{|z|^{2}}\right) + \frac{1}{2^{n}} m_{\alpha}\left( \arctan\frac{4t}{|z|^{2}} \right)\right] \\
& \frac{1}{(|z|^{4}+16t^{2})^{\frac{n}{2}}} f(z,t) dz dt.
\end{align*}
Hence, when $\alpha=0$ since $B_{\frac{1}{2}}\left(\frac{n}{2},\frac{n}{2}\right)=\frac{\Gamma\left(\frac{n}{2}\right)^2}{2\Gamma(n)}$, we recover Folland's fundamental solution for the Heisenberg sublaplacian. Indeed, from the relations between the incomplete Beta and Gamma functions, it follows that
\begin{align*}
<-\Phi_{0},f> & =  4^{n-1}\Gamma\left(\frac{n}{2}\right)^{2} \int\limits_{\mathbb{H}_{n}} \frac{1}{(|z|^{4}+16t^{2})^{\frac{n}{2}}} f(z,t) dz dt.
\end{align*}

\Addresses

\end{document}